\documentclass[11pt,reqno]{amsart}
\usepackage{amsmath,amsfonts,amssymb,amscd,amsthm,amsbsy,epsf}
\newtheorem{thm}{Theorem}

\newtheorem{cor}{Corollary}
\newtheorem*{Cor}{Corollary}

\newtheorem*{Lem}{Lemma}
\theoremstyle{definition}
\newtheorem*{Rem}{Remark}
\def\HH{{\mathbb H}}
\def\real{{\mathbb R}}
\begin{document}
\title[Multilinear Embedding and Hardy's Inequality]
{Multilinear Embedding and Hardy's Inequality}
\author{William Beckner}
\address{Department of Mathematics, The University of Texas at Austin,
1 University Station C1200, Austin TX 78712-0257 USA}
\email{beckner@math.utexas.edu}
\begin{abstract}
Multilinear trace restriction inequalities are obtained for Hardy's inequality. 
More generally, detailed development is given for new multilinear forms for Young's 
convolution inequality, and a new proof for the multilinear Hardy-Littlewood-Sobolev inequality,
and its realization on hyperbolic space. 
\end{abstract}
\dedicatory{for Shanzhen Lu}
\maketitle

Hardy's inequality
\begin{gather}
f \rightsquigarrow   (Tf)(x) = \frac1x \int_0^x f(t)\, dt\ ,\quad x>0 \label{eq1}\\
\noalign{\vskip6pt}
\int_0^\infty \Big| \frac1x \int_0^x f(t)\,dt \Big|^p\, dx 
< \Big( \frac{p}{p-1}\Big)^p \int_0^\infty |f(x)|^p\,dx \ ,\quad 1<p<\infty \nonumber
\end{gather}
lies at the core of understanding for many parts of modern Fourier analysis --- embedding estimates, 
maximal functions, singular integrals, weighted inequalities and measures of uncertainty. 
This operator plays an intrinsic role for understanding how dilation invariance and representation 
for convolution integrals on a group determine optimal estimates.
Recent work by S.~Lu and his collaborators has led to new multilinear and higher-dimensional 
forms for Hardy's inequality. 
Our purpose here is to consider such operators in the context of classical convolution inequalities. 

\section{Multilinear convolution inequalities}

Consider the convolution integral with multi-component decomposition 
$$F(w) =\int_{\real^n\times\cdots\times\real^n} 
G(w-y)\, H(y)\,dy\ ,\qquad w\in \real^{mn}$$
\begin{itemize}
\item[(1)] ``diagonal trace restriction''
$$F(w) \rightsquigarrow F(\underbrace{x,\ldots,x}_{m\text{ slots}}) 
\equiv F(x)\ ,\qquad x\in \real^n$$
\item[(2)] ``multilinear products''
$$F(x) = \int_{\real^{mn}} \prod g_k (x-y_k)\, H(y_1,\ldots,y_m)\, dy\ ;$$
\end{itemize}
the objective is to determine how the components of $G$ and the 
nature of $H$ control the size of $F$.
Here the $g_k$'s will be taken as inputs,  so 
the multilinear map  is given by 
\begin{equation}\label{eq1.1}
H\in L^p (\real^{mn}) \rightsquigarrow F\in L^q (\real^n)\ ,\qquad 1< p < q
\end{equation}

A relatively simple  characterization for this framework can be   
obtained from the classical Young's inequality using a varied set of techniques.

\begin{thm} \label{thm:Young}
For   $1<s_k < 2 \le q <\infty$, $k=1,\ldots, m$  and $m/2 + 1/q = \sum 1/s_k$ 
\begin{equation}\label{eq:thm1}
\|F\|_{L^q(\real^n)} \le C\prod \|g_k\|_{L^{s_k}(\real^n)}  
\|H\|_{L^2 (\real^{mn})}
\end{equation}
$C$ is a constant less than or equal to $1$.
\end{thm} 

\begin{proof} 
Consider the dual problem (primes denote dual exponents, $1/q + 1/q' =1$)
\begin{equation}\label{eq:thm1pf}
\bigg[ \int_{\real^{nm}}  \Big| \int_{\real^n} \prod_k g_k (x-y_k)\, h(x)\,dx\Big|^2 \,dy
\bigg]^{1/2} 
\le C \prod \|g_k\|_{L^{s_k}(\real^n)} \|h\|_{L^{q'}(\real^n)}
\end{equation}
then the left-hand side becomes 
\begin{equation*}
\int_{\real^n\times\real^n} \overline{h(w)} \prod (g_k  *  \tilde g_k) (w-x) h (x)\,dx\,dw
\end{equation*}
with $\tilde f(x)  = \bar f (-x)$.
Here the problem can be equivalently stated as finding the norm for the operator
\begin{equation*}
h\in L^{q'} (\real^n) \rightsquigarrow k * h \in L^q (\real^n)
\end{equation*}
with $k = \prod_k (g_k * \tilde g_k)$ --- a map from a space to its dual. 
So applying Young's inequality, it will suffice to show that $k\in L^r (\real^n)$ with $r=q/2$.
To complete the proof, we need to exhibit a sequence of indices $\{t_k\}$ with $\sum 1/t_k =1$ 
and such that $g_k * g_k \in L^r (\real^n)$ with $r = q t_k/2$.
Then $1/t_k = q [1/s_k - 1/2]$ and $\sum 1/t_k = q [\sum 1/s_k - m/2] =1$ by hypothesis. 
By this argument, the classical value of $C=1$ can be used.
\end{proof}

The structure of inequality \eqref{eq:thm1pf} allows application of the Brascamp-Lieb-Luttinger 
rearrangement   inequality \cite{BLL}  
so that the functions $h$, $\{g_k\}$ can be replaced by their radial decreasing rearrangements 
in order to have an improved inequality. 
Then methods from \cite{BL} determine the sharp constant $C$ by calculating the optimal value 
when the functions are gaussian. 
{From} this explicit calculation, one can see that the analysis does not reduce to iterative processes 
or product function states. 
For example, when $m=2$ extremal functions for \eqref{eq:thm1} will have the form 
$$e^{-\alpha_1 y_1^2} e^{-\alpha_2 y_2^2} e^{-\beta (y_1-y_2)^2}\qquad 
\alpha_1,\alpha_2,\beta >0$$
which is not a product function in the variables $y_1,y_2$.

\begin{cor}\label{cor:optimal}
The optimal value for the constant $C$ in Theorem~\ref{thm:Young} is given by 
\begin{equation}\label{eq:cor-optimal}
C = \bigg[ 2^{-m} (q')^{2/q'} \prod_{k=1}^m (s_k)^{2/s_k} \prod_{\ell=0}^m (\lambda_\ell)^{\lambda_\ell}
\bigg]^{n/4}
\end{equation}
where $\lambda_0 = 1 - 2/q$ and $\lambda_k = 1-2/s'_k$ with $\sum \lambda_\ell=1$.
\end{cor}

\begin{proof} 
As observed above, application of the Brascamp-Lieb-Luttinger rearrangement inequality allows 
determination of the best constant to be attained when all functions are gaussian. 
Consider positive parameters $\{a_k\}$, $k=0,1,\ldots,m$ with 
\begin{gather*}
h(x) = (a_0)^{n/2} e^{-\pi a_0 x^2}\ ,\quad 
g_k (x) = (a_k)^{n/2} e^{-\pi a_k x^2}\ ,\qquad 
k=1,\ldots,m\\
\noalign{\vskip6pt}
\|h\|_{L^{q'} (\real^n)} = (a_0)^{n/2q} (q')^{-n/2q'}\ ;\quad 
\|g_k\|_{L^{s_k} (\real^n)} = (a_k)^{n/2s'_k} (s_k)^{-n/2s_k}
\end{gather*}
Then 
$$(g_k \times g_k) (x) = (a_k/2)^{n/2} e^{-\pi a_k x^2/2}$$
and 
$$C= 2^{-mn/4} (q')^{n/2q'} \prod_k (s_k)^{n/2s_k} \max 
\bigg[ \frac{\prod a_\ell^{\lambda_\ell}}{\sum a_\ell} \bigg]^{n/4}$$
Observe by variation argument for $a_\ell >0$ and $\sum \lambda_\ell =1$
\begin{equation}\label{eq:cor1pf}
\max \bigg[ \frac{\prod a_\ell^{\lambda_\ell}}{\sum a_\ell}\bigg] = \prod \lambda_\ell^{\lambda_\ell}
\end{equation}
\end{proof}

The symmetry in the role of the functions $h$ and $\{ q_k\}$ for the computation of the optimal 
constant suggests that there should be an equivalent formulation of the inequality in 
Theorem~\ref{thm:Young} that displays this symmetry. 
Such a result can be given using the Fourier transform and shows the connection with the 
Hausdorff-Young inequality as suggested by the role of a map from a space to its dual.
And a closer look at the left-hand side in \eqref{eq:thm1pf} shows that the role of $h$ can be 
interchanged with  any single $g_k$ which also displays this symmetry.

\begin{cor}\label{cor:constraints}
With constraints as given in Theorem~\ref{thm:Young}, an equivalent inequality is 
\begin{equation}\label{eq:cor-constraints}
\Big( |\hat h|^2 \times |\hat g_1|^2 * \cdots * |\hat g_m|^2\Big) (0) 
\le C^2 \Big[ \|h\|_{L^{q'}(\real^n)} \prod \|g_k\|_{L^{s_k} (\real^n)} \Big]^2
\end{equation}
with the value of $C$ given in \eqref{eq:cor-optimal}.
\end{cor}

A more general result, though with restricted conditions for an elementary proof, can be given 
for the multilinear map 
$$H \in L^p (\real^{nm}) \rightsquigarrow F\in L^q (\real^n)\qquad 1<p\le q<\infty$$

\begin{thm}\label{thm2}
For $1<p<q<\infty$, $1<s_k < p'/q'$, $k=1,\ldots,m$ and $1/q	 + m/p' = \sum 1/s_k$, 
$m<p'/q'$
\begin{equation}\label{eq:thm2}
\|F\|_{L^q(\real^n)} \le C \prod \|g_k\|_{L^{s_k} (\real^n)} \|H\|_{L^p(\real^{nm})} 
\end{equation}
$C$ is a constant less than or equal to $1$.
\end{thm}

\begin{proof} 
Consider the dual problem (primes denote dual exponents, $1/q + 1/q' =1$)
\begin{equation}\label{eq:thm2pf}
\bigg[ \int_{\real^{nm}} \Big| \int_{\real^n} \prod g_k (x-y_k)h(x)\,dx \Big|^{p'}\,dy\bigg]^{1/p'}
\le C \, \prod \|g_k\|_{L^{s_k}(\real^n)} \|h\|_{L^{q'} (\real^n)}
\end{equation}
Choose the sequence $\beta_k = 1/t_k = q (1/s_k - 1/p')$;  
then  $\sum \beta_k =1$ and $1< t_k < s_k$. 
Apply H\"older's inequality to the inner integral to give
$$\Big| \int \prod g_k (x-y_k) h(x)\,dx\Big|
\le \prod_k \Big[ \big(u_k * |h|\big)(-y_k)\Big]^{\beta_k} $$
for $u_k = |g_k|^{t_k}$. 
Then 
\begin{equation*}
\bigg[ \int_{\real^{nm}} \Big| \int_{\real^n} \prod g_k (x-y_k) h(x)\,dx\Big|^{p'}\, dy\bigg]^{1/p'}
 \le \prod  \Big[\big\|u_k * |h|\,\big\|_{L^{p'/t_k} (\real^n)}\Big]^{\beta_k}
\end{equation*}
Observe that $p'/t_k >1$ since $t_k < s_k < p'/q' < p'$.
Then using Young's inequality
\begin{equation*}
\big\| u_k * |h|\big\|_{L^{p'/t_k}(\real^n)} 
\le \Big[ \|g_k\|_{L^{s_k} (\real^n)}\Big]^{t_k} \| h\|_{L^{q'}(\real^n)}
\end{equation*}
and 
\begin{equation*}
\prod\Big[ \big\|u_k * |h|\big\|_{L^{p'/t_k}(\real^n)} \Big]^{\beta_k} 
\le \Big[ \prod \| g_k\|_{L^{s_k} (\real^n)}\Big] \| h\|_{L^{p'} (\real^n)}
\end{equation*}
which demonstrates inequality \eqref{eq:thm2} with the classical value of 
$C=1$ from Young's inequality.
\end{proof}

There is no overlap between these two theorems since for $p=2$ the second theorem requires 
$m< 2/q' < 2$.
More generally, the method of radial-decreasing equimeasurable rearrangements and reduction 
to gaussian functions (see \cite{BL}) can be applied for the case where $p'$ is  a positive  integer 
at least two. 
For $m=2$ an iterative argument for multidimensional estimates allows a full determination 
in terms of allowed parameters for this multilinear extension of Young's inequality.

\begin{thm}\label{thm:extension}
For $1< p < q < \infty$, $1< s_1,s_2 < p' < \infty$ and $2/p' + 1/q = 1/s_1 + 1/s_2$
\begin{equation}\label{eq:extension}
\|F \|_{L^q (\real^n)} \le C \|g_1 \|_{L^{s_1} (\real^n)} \|g_2\|_{L^{s_2} (\real^n)} 
\| H\|_{L^p(\real^n\times\real^n)}
\end{equation}
$C$ is a constant less than or equal to $1$.
\end{thm}

\begin{proof} 
Consider the dual problem 
\begin{equation}\label{eq:pf-extension}
\begin{split}
&\Bigg[ \int_{\real^n\times\real^n} \Big|\int_{\real^n} g_1 (x-y_1) g_2 (x-y_2) h(x)\,dx\Big|^{p'} 
dy_1\, dy_2\bigg]^{1/p'}\\
\noalign{\vskip6pt}
&\qquad \le C\|g_1\|_{L^{s_1} (\real^n)} \|g_2\|_{L^{s_2}(\real^n)} \|h \|_{L^{q'}(\real^n)}
\end{split}
\end{equation}
Write the left-hand side as 
\begin{equation*}
\begin{split}
&\bigg[ \int_{\real^n} \bigg[ \int_{\real^n} \Big| \int_{\real^n} 
g_1 (x-y_1) w(x)\,dx\Big|^{p'} dy_1\bigg]^{1/p \cdot p'}
dy_2 \bigg]^{1/p'}\\
\noalign{\vskip6pt}
&\qquad = \bigg[ \int_{\real^n} \bigg[ \int_{\real^n}\big| (\tilde g_1 * w)(y_1)\big|^{p'} dy_1\bigg]^{1/p' \cdot p'}
dy_1\bigg]^{1/p'}\\
\noalign{\vskip6pt}
&\qquad \le \| g_1 \|_{L^{s_1} (\real^n)} 
\bigg[ \int_{\real^n} \Big[ \|w\|_{L^r (\real^n)} \Big]^{p'} dy_2\bigg]^{1/p'}\\
\noalign{\vskip6pt}
&\qquad = \|g_1\|_{L^{s_1}(\real^n)} \bigg[ \int_{\real^n} \Big[ \big( |\tilde g_2|^r * |h|^r\big) (y_2)\Big]^{p'/r}
dy_2 \bigg]^{r/p' \cdot 1/r}
\end{split}
\end{equation*}
where $w(x) \equiv g_2 (x-y_2)h(x)$ and $1/s_1 + 1/r -1 = 1/p'$.

It is necessary to show that $1< r < s_2$ and $r/s_2 + r/q' -1 = r/p'$. 
Then $1/r = 1/p' + 1/s'_1 < 1/s_1 + 1/s'_1 =1$ since $s_1 < p'$; and 
$1/r = 1/p' + 1/s'_1 > 1/s_2$ since 
$1/p' +1 = 2/p' + 1/p > 2/p' + 1/q = 1/s_1 + 1/s_2$ as $p< q$.
Finally $r/s_2 + r/q' -1 = r/p'$ is equivalent to 
$1/s_2 + 1/q' - 1/r = 1/p'$ or $1/s_2 + 1/q' - 1/p' = 1/r = 1/p' + 1/s'_1$ 
or $1/s_1 + 1/s_2 = 2/p' + 1-1/q' = 2/p' + 1/q$ (initial condition required by 
dilation invariance).

Under these constraints
\begin{equation*}
\bigg[ \int_{\real^n} \Big[ \big( |\tilde g_2|^r * | h|^r\big) (y_2)\Big]^{p'/r} dy_2\bigg]^{r/p' \cdot 1/r}
\le \|g_2 \|_{L^{s_2} (\real^n)} \|h\|_{L^{q'} (\real^n)}
\end{equation*}
which completes the demonstration of inequality \eqref{eq:pf-extension} 
and the proof of the theorem.
\end{proof}

By extending the argument given above, one obtains a proof by induction for the allowed 
range of parameters for convolution inequalities associated with \eqref{eq1.1} for all values of $m$.

\begin{thm}[Multilinear Young's Inequality]
\label{thm:multilinear} 
For $1< p < q< \infty$, $1< s_k < p'$, $k=1,\ldots,m$ and $1/q + m/p' = \sum 1/s_k$ 
\begin{equation}\label{eq:thm-MYE} 
\|F\|_{L^q (\real^n)} \le \prod \|g_k\|_{L^{s_k} (\real^n)} \| H\|_{L^p (\real^{nm})} 
\end{equation}
\end{thm}

\begin{proof} 
Consider the dual problem 
\begin{equation}\label{eq:dualprob}
\bigg[ \int_{\real^{nm}} \Big| \int_{\real^n} \prod g_k (x-y_k) h (x) dx\Big|^{p'} dy\bigg]^{1/p'}
\le \prod \|g_k\|_{L^{s_k} (\real^n)} \|h\|_{L^{q'}(\real^n)}
\end{equation}
The method used in Theorem~\ref{thm:extension} allows the construction of an argument 
so that the inequality for $m$ copies of $\real^n$ is reduced to proving the inequality for 
the case of $m-1$ copies of $\real^n$. 
Then the theorem is established from using Theorem~\ref{thm:extension} where the case 
$m=2$ is proved. 
Order the values of $\{s_k\}$ so that $1< s_1 \le s_2 \cdots \le s_m < p'$. 
Start with the integral
\begin{gather*}
\bigg[ \int \Big| \int_{\real^n} \prod g_k (x-y_k) h(x) dx \Big|^{p'} dy_1\bigg]^{1/p'}
= \bigg[ \int \Big| \int_{\real^n} g_1 (x-y_1) w(x) dx\Big|^{p'} dy_1\bigg]^{1/p'}\\
\noalign{\vskip6pt}
\le \|g_1 \|_{L^{s_1}(\real^n)} \|w\|_{L^r (\real^n)}
\end{gather*}
where 
$$w(x) = \prod_{k\ge 2} g_k (x-y_k) h (x)\ ,\qquad \frac1{s_1} + \frac1r -1 = \frac1{p'}\ .$$
Now it it necessary to show that $1< r < s_2 \le s_k < p'$. 
First, $1/r = 1/p' + 1/s'_1 < 1/s_1 + 1/s'_1 =1$; 
second, $m/p' + 1/q = \sum 1/s_k$ and 
\begin{gather*}
\frac2{p'} + \frac1q = \frac1{s_1} + \frac1{s_2} + \sum \left( \frac1{s_k} - \frac1{p'}\right) > \frac1{s_1}
+ \frac1{s_2}\\
\noalign{\vskip6pt}
\frac1r = \frac1{p'} + \frac1{s'_1}  = \frac2{p'} + \frac1p - \frac1{s_1} > \frac2{p'} + \frac1q - \frac1{s_1} 
> \frac1{s_2}
\end{gather*}
so $r < s_2 \le s_k < p'$.
So writing out this representation
\begin{equation*}
\bigg[ \int_{\real^{(m-1)n}} \Big[ \|w\|_{L^r (\real^n)}\Big]^{p'} dy\bigg]^{1/p'}
= \bigg[ \int_{\real^{(m-1)n}} \Big| \int 
\mathop{\prod}\nolimits' 
u_k (x-y_k) v(x) dx\Big|^{p'/r} dy \bigg]^{1/p'}
\end{equation*}
where $u_k = |g_k|^r$ and $v = |h|^r$, and for $t_k = s_k/r$, $\alpha = p'/r$, 
$\sigma' = q'/r$, then the corresponding constraint becomes $(m-1)/\alpha + 1/\sigma = \sum' 1/t_k$
which by using $1/r = 1/p' + 1/s'_1$ becomes the $m^{th}$-level constraint 
$m/p' + 1/q = \sum 1/s_k$ (here $\sum'$ indicates that some terms are not included). 
Since the number of $y$-integrations has now been reduced by one, this process can be repeated 
until one reaches $m=2$ which is already demonstrated in Theorem~\ref{thm:extension}. 
\end{proof}

Because of the iterative nature of the proof of Theorem~\ref{thm:multilinear}, it would be 
natural to consider whether this product structure could determine the nature of the optimal constant.
But since the extremal functions will not be simple product functions, such an argument would not be 
sufficient. 
However this method of iterative reduction does enable an extension to include Riesz potentials 
with a resulting multilinear Hardy-Littlewood-Sobolev inequality. 
For the case $p=2$, this problem was treated in \cite{Beckner-MRL}.
Still a different approach is given in \cite{Beckner-Essays} (see page 48) where conformal 
invariance is used. 

\begin{thm}[Multilinear Hardy-Littlewood-Sobolev Inequality]
\label{thm:HLS}
Let 
\begin{equation}\label{eq:HLS-a}
F(x) = \int_{\real^{nm}} \prod |x-y_k|^{-\lambda_k} H(y_1,\ldots, y_m)\,dy\ ,\qquad 
\lambda_k = n/s_k
\end{equation}
for $1< p < q < \infty$, $1< s_k < p'$, $k=1,\ldots, m$, $1/q + m/p' = \sum 1/s_k$; then 
\begin{equation}\label{eq:HLS-b} 
\|F \|_{L^q (\real^n)} \le C \|H \|_{L^p (\real^{nm})}
\end{equation}
\end{thm}

\begin{proof}
Look at the dual inequality and use the same induction argument as for 
Theorems~\ref{thm:extension} and \ref{thm:multilinear} 
while applying the standard inequality for bounds for the Riesz potentials on $\real^n$.
\end{proof}

\begin{Cor}
Let 
$$F(x) = \int_{\real^{nm}} \prod g_k (x-y_k) H (y_1,\ldots,y_m)\,dy$$
for $1 < p < q < \infty$, $1< s_k < p'$, $k=1,\ldots,m$ and $1/q + m/q' = \sum 1/s_k$; then 
$$\|F\|_{L^q (\real^n)} \le C \prod \|g_k\|_{L_{s_k,\infty} (\real^n)} 
\|H\|_{L^p (\real^{nm})}$$
\end{Cor}

\begin{proof} 
Return to the proof of Theorem~\ref{thm:multilinear}, and the inductive step 
\begin{equation*}
\begin{split}
&\bigg[ \int\Big| \int_{\real^n} \prod g_k (x-y_k) h (x) \,dx\Big|^{p'} dy_1\bigg]^{1/p'}
= \bigg[ \int \Big| \int_{\real^n} \mkern-12mu
g_1 (x-y_1) w(x)\,dx \Big|^{p'} dy_1\bigg]^{1/p'}\\
\noalign{\vskip6pt}
&\quad 
\le \bigg[ \int \Big| \int_{\real^n}\mkern-12mu  
 g_1^* (x-y_1) w^* (x)\,dx\Big|^{p'} dy_1\bigg]^{1/p'} \mkern-8mu
\le \|g_1 \|_{L_{s_1,\infty}(\real^n)} \bigg[ \int \Big| \! \int_{\real^n} \mkern-12mu
|x-y_1|^{-\lambda_1} w^* (x)dx\Big|^{p'}
dy_1\bigg]^{1/p'}\\
\noalign{\vskip6pt}
&\quad
\le C\|g_1 \|_{L_{s_1,\infty}} \| w^* \|_{L^r (\real^n)} 
= C \|g_1\|_{L_{s_1,\infty}(\real^n)} \|w\|_{L^r (\real^n)}
\end{split}
\end{equation*}
Repeat this inductive step for $g_2$ and continue the iterative steps.
Here $g^*$ denotes the non-negative equimeasurable radial decreasing rearrangement of $g$,
and 
$$g^* (x) \le C\| g\|_{L_{s,\infty }(\real^n)} |x|^{-n/s}$$
for $g$ in the weak-Lebesgue space $L_{s,\infty} (\real^n)$.
\end{proof}

\section{Diagonal trace restriction for Hardy's inequality}

The multilinear convolution structure described above can be applied to Hardy's inequality:
\begin{gather}
(Tf (x) = \frac1x \int_0^x f(t)\,dt\ ,\qquad x>0\nonumber\\
\noalign{\vskip6pt}
\|Tf\|_{L^p(0,\infty)} < \frac{p}{p-1} \|f\|_{L^p(0,\infty)} \ ,\qquad 1< p<\infty \label{sec2-eq6}\\
\noalign{\vskip6pt}
\|x^{1/r} Tf\|_{L^q (0,\infty) }
\le C_{p,q} \|f\|_{L^p (0,\infty)}\ ,\qquad 1<p<q<\infty\label{sec2-eq7}\\
\noalign{\vskip6pt}
\frac1r = \frac1p - \frac1q\ ,\quad C_{p,q} = (p'/q)^{1/q} 
\left[ \frac{ q/r \   \Gamma (r)}{\Gamma (r/q ) \Gamma (r/q')}\right]^{1/r}  
\nonumber
\end{gather}
The constants are best possible (see \cite{H, HL, Bliss, HLP}). 
Equality is attained for functions of the form 
$$f(t) = A (1+ct^{q/r})^{-q/(q-p)}$$
These operators extend by tensor product to the manifold $  [0,\infty)^m$: $T\otimes T\otimes 
\cdots \otimes T$ acts on functions defined on $[0,\infty)^m$ and $S\otimes S\otimes \cdots\otimes S$
with $S = x^{1/r} T$ similarly acts on such functions. 
By iteration one obtains:

\begin{thm}\label{thm3}
For $1<p<q<\infty$
\begin{gather}
\|(T\otimes\cdots\otimes T) H\|_{L^p [0,\infty)^m} < \left( \frac{p}{p-1}\right)^m \|H\|_{L^p [0,\infty)^m}
\label{sec2-eq8}\\
\noalign{\vskip6pt}
\|(S\otimes\cdots\otimes S) H\|_{L^q [0,\infty)^m} \le [C_{p,q}]^m \| H\|_{L^p [0,\infty)^m}
\label{sec2-eq9}
\end{gather}
\end{thm}

\begin{proof}
This extension to higher dimensions is obtained by iteration of one-dimensional estimates (see
\cite{Beckner-Annals75}). 
The constants are best possible from considering product functions.
\end{proof}

Let $\Lambda_m$ denote the first quadrant in $\real^m$; namely, $[0,\infty)^m$. 
For a non-negative function $f$ defined on $\Lambda_m$, let $f_{\#}$ denote the 
equimeasurable rearrangement of $f$ as a radial decreasing function on $\Lambda_m$. 
Then 

\begin{Lem} 
\begin{equation}\label{eq:sqr} 
\int_{\Lambda_m} f(x) g(x) \,dx \le \int_{\Lambda_m} f_{\#} (x) g_{\#} (x)\,dx 
\end{equation}
\end{Lem}

\begin{proof}
Extend $f$ and $g$ to all of $\real^m$ by copying the functions to all equivalent sectors; 
then 
$$\int_{\Lambda_m} f(x) g(x) \,dx
= 2^{-m} \int_{\real^m} f(x) g(x) \,dx  
\le 2^{-m} \int_{\real^m} f^* (x) g^* (x)\,dx 
= \int_{\Lambda_m} f_{\#} (x) g_{\#} (x)\,dx $$
\end{proof}

Such a rearrangement is optimal for any sector in $\real^m$ having vertex at the origin which you can 
replicate to pave $\real^m$.

\begin{thm}\label{thm:rearrangement}
For $1< p \le q < \infty$ and $\gamma = m/p - 1/q$ with $H\in L^p (\Lambda_m)$ and 
$\psi (x) \break = x^\gamma [(T\otimes\cdots \otimes T) H ](x,\ldots,x)$
\begin{equation}\label{eq:rearrange}
\|\psi \|_{L^q [0,\infty)} \le C \|H\|_{L^p (\Lambda_m)}
\end{equation}
\end{thm}

\begin{proof} 
By applying the rearrangement lemma for reduction to radial functions and rescaling, 
a simple bound for this diagonal trace restriction inequality is obtained from the weighted 
Hardy inequality \eqref{sec2-eq7}. 
\begin{equation}\label{eq:simplebound}
\psi (x) \le x^{-m/p' - 1/q} \int_{B_x\cap \Lambda_m} f^* (t) \,dt 
= 2^{-m} \text{vol}(B) \int_0^{ax} f^* (u) u^{m-1}\,du
\end{equation}
where $u = |t|$ and $B_x$ is a ball centered at the origin with $\text{vol}(B_x) = 2^m x^m$;
that is, the ball has radius $ax$ so that $a^m = 2^m/\text{vol}(B)$ with $g(v) = f^* (av^{1/m})$, 
$v = |t|^m$ and $B$ being the unit bball of radius~1.
Then after rescaling to remove the factor $a$, changing variables from $t\in \Lambda_m$ to 
$v= |t|^m$, and clearing away factors of spherical volume, inequality \eqref{eq:rearrange} is 
equivalent to the weighted one-dimensional Hardy inequality 
\begin{equation}\label{eq:weighted} 
\bigg[ \int_0^\infty \Big( x^{-1/p' - 1/q} \int_0^x g(v)\,dv \Big)^q \,dx\bigg]^{1/q}
\le C_{p,q} \bigg[ \int_0^\infty |g(v)|^p \,dv\bigg]^{1/p}
\end{equation}
\end{proof}

To have a better sense of the method used in this argument, it will be useful to look at an 
$n$-dimensional variant of Hardy's inequality which was suggested in \cite{CG95}. 
On $\real^n$ let 
$$W(x) = \frac1{\text{vol}(B)} \frac1{x^n} \int_{|t|<x} f(t)\,dt\ ,\qquad x>0$$
where now $W(x)$ is considered as a function on $\real^n$ with $x$ being the radial 
variable and $B$ is the unit ball in $\real^n$.

\begin{thm}\label{thm:optimal}
For $1 < p \le q <\infty$ and $\gamma = n/p - n/q = n/r$
\begin{equation}\label{eq:optimal} 
\|x^\gamma W\|_{L^q (\real^n)} \le C\|f\|_{L^p (\real^n) }
\end{equation} 
The optimal constant is given by $(\text{\rm vol}(B)]^{-1/r} C_{p,q}$ 
(see inequality \eqref{sec2-eq7}).
\end{thm}

\begin{proof} 
Observe that $W(x,f) \le W(x,f^*)$ and choose new variables in terms of $|t|^n$; then simplify to 
obtain \eqref{eq:optimal}.
\end{proof}

Because of dilation invariance and reduction to one dimensional estimates, inequalities 
in this section have equivalent representations as convolution inequalities on the 
multiplicative group $\real_+$ where Haar measure is given by $1/t\  dt$ and convolution 
on the group is realized as 
$$(g*h) (x) =  \int_G g(y) h(y^{-1} x)\,dy$$
with $dy$ denoting Haar measure on $G$.
Then bounds are obtained using Young's inequality though for $p\ne q$ the constants will not be 
optimal. 

\begin{thm}\label{thm:inequality}
For $1< p \le q <\infty$ with $1/r = 1/p - 1/q$ the following two inequalities are equivalent:
$$
\|x^{1/r} Tf\|_{L^q (0,\infty)} 
\le C_{p,q} \|f\|_{L^p (0,\infty)} 
\eqno{\eqref{sec2-eq7}}$$
and 
\begin{equation}\label{eq:inequality2}
\| g * h \|_{L^q (\real_+)} 
\le 
C_{p,q}  
\|h\|_{L^p (\real_+)}
\end{equation}
with 
$g(y) = y^{1/p'} \chi_{[0,1]} (y)$.
\end{thm}

\begin{proof} 
In \eqref{sec2-eq7} 
set $f(t) = t^{-1/p} h (1/t)$ and simplify terms to obtain 
\eqref{eq:inequality2}.
\end{proof}

The sharp constants for convolution on the multiplicative group $\real_+$ transfer directly from 
the additive translation group $\real$ (see \cite{Beckner-Annals75}. 

\begin{Lem}[Young's inequality on $\real_t$] 
For $1 < p\le q < \infty$ and $1/q = 1/p - 1/s'$
\begin{equation}\label{eq:lem-Young} 
\| f * g\|_{L^q (\real_t)} \le D_{p,q} \|f\|_{L^p (\real_t)} \|g\|_{L^s (\real_t)}
\end{equation}
\end{Lem}

But one can reverse this map from $\real$ to $\real_+$ and examine the nature of the 
convolution inequality \eqref{eq:inequality2} and the corresponding Hardy inequality in terms of 
analysis on the entire line. 
A question of particular interest is to determine sharp estimates in Young's inequality when one of 
the functions is specified and the map is from a space to its dual:
\begin{equation}\label{eq:Young_dual}
f \rightsquigarrow \varphi * f\ ,\quad | \varphi * f\|_{L^{p'}(\real^n)} 
\le C_p \|f\|_{L^p (\real^n)}\ ,\qquad 
1< p < 2
\end{equation}

Two intrinsic cases on $\real^n$ are: 
a)~gaussian $\varphi (x) = e^{-|x|^2}$ (see \cite{Beckner-Annals75});
b)~Riesz potential $\varphi (x) = |x|^{-2n/p'}$ (see \cite{Lieb83}). 
Hardy's inequality allows the extension of this program to include the exponential density 
$\varphi (x) = e^{-|x|}$ for one dimension.

\begin{thm}[Young's inequality for an exponential density] 
\label{thm:Young_density}
For $f\in L^p(\real)$, $1<p<2$, $1/p + 1/p' =1$ and $\varphi (x)=e^{-|x|}$
\begin{gather}
\|\varphi * f\|_{L^{p'}(\real)} \le C_p \|f\|_{L^p(\real)} 
\label{eq:young}\\
\noalign{\vskip6pt}
C_p = \Big( \frac{p'}2\Big)^{2/p}
\left[ \frac{\Gamma (\frac{2p}{2-p})}
{\Gamma (\frac2{2-p}) \Gamma (\frac{p}{2-p})} 
\right]^{2/p\ -\ 1}\notag
\end{gather}
\end{thm}

\begin{proof} 
{From} the Hardy inequality for $q=2$:
\begin{equation*}
\int_0^\infty \Big| \int_0^s g(t)\,dt\Big|^2 s^{-1\ -\ 2/p'}\,ds 
\le K^{2/p} \bigg[ \int_0^\infty |g|^p\,ds\bigg]^{2/p} \ ,
\qquad K = [C_{p,2}]^p
\end{equation*}
First, set $g(t) = h(t) t^{-1/p}$, and then set $t= e^x$
\begin{gather*}
\int_0^\infty \Big| \int_0^s h(t) (t/s)^{1/p'} \frac1t\,dt\Big|^2 
\frac1s\,ds 
\le K^{2/p} \bigg[ \int_0^\infty |h|^p \frac1s \,ds\bigg]^{2/p}\\
\noalign{\vskip6pt}
\int_{-\infty}^\infty \Big| \int_{-\infty}^x h(y) e^{-(x-y)/p'}\,dy\Big|^2\,dx
\le K^{2/p} \bigg[ \int_{-\infty}^\infty |h|^p\,dx \bigg]^{2/p}\\
\noalign{\vskip6pt}
\Big| \int_{\real\times\real} h(u) e^{-2/p'|u-v|} h(v)\,du\,dv\Big| 
\le \frac2{p'} K^{2/p} \bigg[\int_{\real} |h|^p\,dx\bigg]^{2/p}\\
\noalign{\vskip6pt}
C_p = \Big(\frac{p'}2\Big)^{2/p\ -\ 1} K^{2/p} 
= \Big( \frac{p'}2\Big)^{2/p} 
\bigg[\frac{\Gamma (\frac{2p}{2-p})} 
{\Gamma(\frac{2}{2-p}) \Gamma (\frac{p}{2-p})}\bigg]^{2/p\ -\ 1}
\end{gather*}
An extremal function for inequality~\eqref{eq:young} is given by 
$$f(x) = \Big[\text{cosh} \big[p'x/(4p\delta)\big]\Big]^{-\delta}\ ,\qquad 
\delta = \frac2{2-p}\ .$$
\end{proof}

Following the direction defined by the calculation for the optimal constant in 
Theorem~\ref{thm:Young} and the principal results in \cite{Beckner-MRL}, the question 
of finding the optimal constant for $p=2$ in Theorem~\ref{thm:rearrangement} is examined. 
The importance of such estimates lies with providing insight into the structure of ``maps from 
a space to its dual''. 
Here there are two independent calculations corresponding to the cases $p = 2 = q'$, $m\ge2$ 
and $p=2>q'$. 
The results for the Hardy operator are analogous to the cases that use Riesz potentials and 
extend Pitt's inequality and the Hardy-Littlewood-Sobolev inequality for multilinear embedding 
in \cite{Beckner-MRL}.

\begin{thm}\label{thm:independent}
For $m\ge2$, $2\le q < \infty$ and $\gamma = m/2 - 1/q$ with $H\in L^2(\Lambda_m)$ and 
$\psi (x)  = x^\gamma [(T\otimes \cdots \otimes T) H ](x,\ldots x)$
\begin{equation}\label{eq:independ1}
\|\psi \|_{L^q (0,\infty)} \le C\|H\|_{L^2 (\Lambda_m)}
\end{equation}
where for $q=2$
\begin{equation}\label{eq:independ2} 
C = 2/\sqrt{m}
\end{equation}
and for $q>2$,  $C$ is determined by the optimal constant for the weighted Hardy inequality.
\end{thm}

\begin{proof} 
As in the calculations in the previous section, consider the dual problem where 
$\alpha = m/2 + 1/q = \gamma + 2/q$
\begin{equation*}
\begin{split}
& \int_{\Lambda_m} \Big| \int_{x > \{ t_1,\ldots,t_m\}} \mkern-48mu 
x^{-\alpha} h (x)\,dx \Big|^2 \, dt_1 \cdots dt_m\\
\noalign{\vskip6pt}
&\qquad 
= m! \int_{\Lambda_m} \Big| \int_{x > t_1 > \cdots > t_m} \mkern-48mu 
x^{-\alpha} h(x)\,dx\Big|^2\, dt_1 \cdots dt_m\\
\noalign{\vskip6pt}
&\qquad 
= 2m! \int_{\Lambda_m} \iint_{w > x > t_1 > \cdots > t_m}\mkern-48mu 
x^{-\alpha} w^{-\alpha} h(x) h(w)\,dx\, dw\,dt_1\cdots dt_m\\
\noalign{\vskip6pt}
&\qquad 
= 2\int_0^\infty h(w) w^{-\alpha} \bigg[ \int_0^w x^\gamma h(x) \,dx\bigg]\,dw 
\le C^2 \Big[ \|h\|_{L^{q'} [0,\infty)} \Big]^2
\end{split}
\end{equation*}
which is equivalent to the inequality 
\begin{equation}\label{eq:independ-pf}
\| k * g \|_{L^q (\real_+)} \le C^2/2 \ \|g\|_{L^{q'} (\real_+)}
\end{equation}
where $k(t) = t^{-m/2} \chi_{[1,\infty)} (t)$.
Then for $q=2$, the best constant is given by the $L^1 (\real_+)$ norm of $k$ 
which has the value $2/m$ with $C = 2/\sqrt{m}$.
To determine the optimal constant for the case $q>2$, one must relate bounds for the operator
\begin{equation}\label{eq:independ-pf2}
h\in L^{q'} (0,\infty) \rightsquigarrow w^{-\sigma - 1/q} \int_0^w x^{\sigma - 1/q} 
h(x)\, dx \in L^q (0,\infty)
\end{equation} 
to those for the weighted Hardy inequality \eqref{sec2-eq7} 
in the case of dual exponents and $\sigma = m/2$ by making a simple change of variables. 
Note that the dilation character of the symmetric form 
\begin{equation*}
\int_0^\infty h(w) w^{-\sigma - 1/q} \bigg[ \int_0^w x^{\sigma -1/q} h(x)\,dx\bigg] \,dw
\end{equation*}
is measured by the index $2/q$ which corresponds to the example of Riesz potentials and 
the Hardy-Littlewood-Sobolev inequality on $\real^n$ where the index would be $2n/q$.
Set
$$h(x) = x^\beta u (x^\delta)\ ,\qquad q' \beta = \delta -1$$
Then 
$$\|h\|_{L^{q'} (0,\infty)} = \delta^{-1/q'} \|u\|_{L^{q'}(0,\infty)}$$
and 
\begin{equation*}
\begin{split} 
\int_0^\infty h(w) w^{-\sigma - 1/q} \bigg[ \int_0^2 x^{\sigma - 1/q} h(x)\,dx\bigg]\, dw
& = \delta^{-2} \int_0^\infty u(w) w^{-2/q} \bigg[ \int_0^w u(x)\, dx\bigg]\,dw\\
\noalign{\vskip6pt}
& = \delta^{-2} \int_0^\infty u(w) w^{1/q' - 1/q} (Tu)(w)\,dw
\end{split}
\end{equation*}
for the choice $\delta = q\sigma$. 
Then $C^2$ for the constant appearing in the inequality \eqref{eq:independ1} 
when $q>2$ is given by 
\begin{equation*}
2(2/mq)^{2/q} \sup \Big\{ \|x^{1/r} Tu\|_{L^q (0,\infty)} \Big\slash \|u\|_{L^{q'}(0,\infty)} \Big\}
= 2 (2/mq)^{2/q} C_{q',q}
\end{equation*}
with $1/r = 1/q' - 1/q$ and $2<q<\infty$. 
Then the constant $C$ for inequality \eqref{eq:independ1} is given by 
\begin{equation} \label{eq:independ-pf3}
C = \sqrt2\ (2/mq)^{1/q} \sqrt{C_{q',q}}
\end{equation}
\end{proof}

A functional integral that has similar properties to the Hardy functional is the 
Hilbert integral:
\begin{gather}
f \rightsquigarrow I(f) (x) = \int_0^\infty \frac1{x+y} f(y)\,dy\ ,\qquad x>0 \label{eq:funct-int}\\
\noalign{\vskip6pt}
\|I(f)\|_{L^p (0,\infty)} \le \Gamma (1/p) \Gamma (1/p') \|f\|_{L^p (0,\infty)}\ ,\qquad 1<p<\infty 
\nonumber
\end{gather}
Observe that the Hilbert integral is controlled by the Hardy operator: 
$I(f) \le T(f) + T^* (f)$ where $T^*$ denotes the adjoint for $T$.
Such simple estimates allow one to recover elementary inequalities such as $\theta (1-\theta) 
\Gamma (\theta) \Gamma (1-\theta) \le 1$ and   provide a check that ``calculated constants 
make sense.'' 
Because of the associated role of exponential densities when dilation-invariant operators 
are transferred from the real half-line to the multiplicative group $\real_+$ and then to the 
additive group $\real$ (see proof of Theorem~\ref{thm:Young_density}), it is instructive 
to apply the methods used here and in \cite{Beckner-MRL} to calculate the precise norm for 
the diagonal trace map corresponding to the Hilbert integral from $L^2(\Lambda_m)$ to 
$L^2 (0,\infty)$. 

\begin{thm}\label{thm:precise_norm}
For $m\ge 2$ and $\gamma = (m-1)/2$ with $H\in L^2(\Lambda_m)$ and \newline
$\psi (x) = x^\gamma [(I\otimes\cdots\otimes I)H] (x,\ldots,x)$
\begin{gather} 
\|\psi \|_{L^2 (0,\infty)} \le C\|H\|_{L^2(\Lambda_m)} \label{eq:precise_norm}\\
\noalign{\vskip6pt}
C = 2 \bigg[ \int_0^\infty \left[ \frac{x}{\sinh x}\right]^m \,dx \bigg]^{1/2}\ .\nonumber
\end{gather}
\end{thm}

\begin{proof} 
As with the proof of the preceding theorem, consider the dual problem 
\begin{align*}
&\int_{\Lambda_m} \Big| \int_0^\infty x^\gamma \prod \frac1{x+y_k} \, h(x)\,dx\Big|^2\,
dy_1 \cdots dy_m\\
\noalign{\vskip6pt}
&\qquad 
= \int_0^\infty \int_0^\infty h(x) x^\gamma \bigg[ \int_0^\infty \frac1{x+y}\, \frac1{w+y}\,dy
\bigg]^m w^\gamma h(w)\, dx\,dw\\
\noalign{\vskip6pt}
&\qquad 
= \int_0^\infty \int_0^\infty h (x) \left[ \frac{(xw)^{(m-1)/2}}{|x-w|^m} \Big| \ln \frac{x}w \Big|^m
\right] h(w)\,dw\\
\noalign{\vskip6pt}
&\qquad 
\le C^2 \Big[ \|h\|_{L^2 (0,\infty)} \Big]^2
\end{align*}
First set $h(x) = x^{-1/2} g(x)$ and then change variables $x= e^t$ and rescale by a factor
of 2 to obtain a convolution inequality on the real line 
$$\|\varphi * f\|_{L^2 (\real)} \le C^2 \|f\|_{L^2 (\real)}\ ,\qquad \varphi (x) = 2[x/\sinh x]^m$$
with 
$$C^2 = \|\varphi\|_{L^1 (\real)} = 4 \int_0^\infty [x/\sinh x]^m\,dx\ .$$
\end{proof}

\section{Diagonal Trace Restriction for a Multilinear Fractional Integral}

Modeled on results such as equation \eqref{eq1.1} above plus ideas from \cite{Beckner-MRL}, 
\cite{KS99} (see Theorem~1 and Lemma~7) and \cite{SL}, 
a diagonal trace restriction inequality is obtained for a multilinear fractional integral.

\begin{thm}\label{thm:diagonal}
For $H\in L^p (\real^{mn})$, $\rho (x_1,\ldots,x_m) = \sum |x_k|$, $x_k\in \real^n$ and 
$$F(x) = \int_{\real^{mn}} \rho (x-y_1,\ldots x-y_m)^{-(mn-\alpha)} 
H(y_1,\ldots, y_m)\,dy$$
with $\alpha/n = m/p - 1/q$; then for $q\ge p$
\begin{equation}\label{eq:diagonal-1}
\|F\|_{L^q (\real^n)} \le C \|H \|_{L^p (\real^{mn})} 
\end{equation}
For $p =2$, the optimal constant is
\begin{equation}\label{eq:diagonal-2} 
C =  \sqrt{B} \  \pi^{n/2q} \left[ \frac{\Gamma (\frac{n}2 - \frac{n}q)}{\Gamma (n)}\right]^{1/2} 
\left[ \frac{\Gamma (n)}{\Gamma (\frac{n}{q'})}\right]^{1/q'}
\end{equation}
\begin{equation*}
B = \int_{\real^{mn}} \Big( \sum |\hat\eta - y_k|\Big)^{-(mn-\alpha)} 
\Big( \sum |y_k|\Big)^{-(mn-\alpha)} dy\ ,\qquad |\hat \eta |=1
\end{equation*}
\end{thm}

\begin{proof} 
By using a simple comparison between sums and products, the first part of this theorem 
follows from the multilinear Hardy-Littlewood-Sobolev inequality (Theorem~\ref{thm:HLS} above). 
For the case $p=2$, use duality to consider 
\begin{equation}\label{eq:comparison}
\bigg[ \int_{\real^{nm}} \Big| \int_{\real^n} \rho (x-y_1,\ldots, x-y_m)^{-(mn-\alpha)} 
h(x)\,dx \Big|^2\,dy \bigg]^{1/2} 
\le C\| h\|_{L^{q'} (\real^n)}
\end{equation}
Then 
\begin{equation*}
\begin{split}
&\int_{\real^{mn}} \rho (x-y_1,\ldots, x-y_m)^{-(mn-\alpha)} 
\rho (w-y_1,\ldots, w-y_m)^{-(mn-\alpha)} dy\\
\noalign{\vskip6pt}
&\qquad = \int_{\real^{mn}} \rho (x-w-y_1,\ldots, x-w-y_m)^{-(mn-\alpha)} 
\rho (y_1,\ldots, y_m)^{-(mn-\alpha)} dy\\
\noalign{\vskip6pt}
&\qquad = |x-w|^{-(mn-2\alpha)} 
B = |x-w|^{-2n/q} B
\end{split}
\end{equation*}
where 
\begin{equation*}
B = \int_{\real^{mn}} \rho (\hat\eta - y_1,\ldots, \hat \eta -y_m)^{-(mn-\alpha)} 
\rho (y_1,\ldots, y_m)^{-(mn-\alpha)} dy\ ,\qquad \hat\eta \in S^{n-1}
\end{equation*}
Observe that $B$ does not depend on the value of the unit vector $\hat\eta$. 
So the dual expression becomes 
\begin{equation*}
B \int_{\real^n\times\real^n} h(x) |x-w|^{-2n/q} h(w)\,dx\,dw 
\le C^2 \Big[ \| h\|_{L^{q'} (\real^n)}\Big]^2
\end{equation*}
Using the sharp constant for the classical Hardy-Littlewood-Sobolev inequality 
\begin{equation*}
C^2 = B\ \pi^{n/q} \ \frac{\Gamma (\frac{n}2 - \frac{n}q)}{\Gamma (\frac{n}{q'})} \ 
\left[ \frac{\Gamma (n)}{\Gamma (\frac{n}{q'})}\right]^{2/q' \, -\, 1}
\end{equation*}
\end{proof}

The dilation character of the above argument easily transfers to a more general statement.
Consider a non-negative function $G(x_1,\ldots,x_m)$ defined on $\real^n\times\cdots\times 
\real^n \simeq \real^{mn}$ that is invariant under a uniform rotation:
$$G(Rx_1, \ldots , Rx_m) = G(x_1,\ldots, x_m)\ ,\qquad R \in SO(n)$$
and with the dilation property --- there exists a sequence $(\beta_1,\ldots,\beta_m)$, $\beta_k\ge0$ 
and $\sum \beta_k = mn-\alpha = mn/p' + n/q$, $q>p$ such that for $\delta_k >0$
$$G(\delta)_1 x_1 ,\ldots, \delta_m x_m) = \prod \delta_k^{-\beta_k} G (x_1,\ldots, x_m)$$

\begin{Cor}
For $H \in L^2 (\real^{mn})$ and 
$$F(x) = \int_{\real^{mn}} G (x-y_1,\ldots, x-y_m) H (y_1,\ldots ,y_m)\,dy$$
with $\alpha = m/2 - 1/q$; then for $q>2$
\begin{equation}\label{eq:Cor-dilation}
\|F\|_{L^q (\real^n)} \le C\|H \|_{L^2 (\real^{mn})}
\end{equation} 
with the optimal constant 
\begin{gather*}
C = \sqrt{B}\ \pi^{n/2q} \left[ \frac{\Gamma (\frac{n}2 - \frac{n}q)}{\Gamma (n)}\right]^{1/2} 
\left[ \frac{\Gamma (n)}{\Gamma (\frac{n}{q'})}\right]^{1/q'}\\
\noalign{\vskip6pt}
B = \int_{\real^{mn}} G (\hat\eta - y_1,\ldots, \hat\eta- y_m) G (y_1,\ldots, y_m)\,dy\ ,\qquad |\eta|=1
\end{gather*}
if this value for $B$ is finite. 
\end{Cor}

\begin{proof} 
Follow directly the previous argument given above for Theorem~\ref{thm:diagonal}.
\end{proof}

\section{Multilinear Integrals and Rearrangement}

A central focus for the methods developed here has been: 
(1)~demonstration of estimates for multilinear multidimensional forms that arise from duality 
arguments, and 
(2)~calculation of optimal constants or norms in select cases. 
Classic examples beyond Hardy's inequality involve:
\begin{itemize}
\item[{}] Young's inequality:\quad $f\rightsquigarrow f*g$
\item[{}] Hardy-Littlewood-Sobolev inequality:\quad $f\rightsquigarrow |x|^{-\lambda} * f$
\item[{}] Stein-Weiss fractional integral:\quad $f\rightsquigarrow |x|^{-\beta} (|x|^{-\lambda} * |x|^{-\alpha}f)$
\end{itemize}
Representation of norm estimates in terms of multilinear integral functionals allows application 
of equimeasurable rearrangement techniques to obtain improved inequalities.
Insight was first gained from the Riesz-Sobolev Rearrangement Inequality:
\begin{equation}\label{eq:RS-rearrangement}
\int_{\real^n\times\real^n} f(x) g(x-y) h(y)\,dx\,dy
\le \int_{\real^n\times\real^n} f^* (x) g^* (x-y) h^* (y) \,dx\,dy
\end{equation}
where $f,g,h\ge 0$ and $f^*$ denotes the equimeasurable radial decreasing rearrangement of 
$f$ on $\real^n$. 
The proof based on the Brunn-Minkowski inequality allowed a much more general formulation 
given by Brascamp, Lieb and Luttinger \cite{BLL}: for non-negative functions $f_k$, $k=1,\ldots,M$
\begin{equation}\label{eq:BLL}
\int_{\real^n\times \cdots\times \real^n} \prod \ f_k \Big( \sum a_{kj} x_j\Big)\, dx_1 \cdots dx_m
\le \int_{\real^n\times \cdots\times\real^n} \prod f_k^* \Big( \sum a_{kj} x_j\Big)\,dx_1\cdots dx_m
\end{equation}
Some restrictions will be necessary to ensure finite values for the integrals.
Brascamp and Lieb used this rearrangement improvement combined with an iterative 
high-dimensional asymptotic argument to show that the optimal bound for this integral 
under constraints given by Lebesgue space norms for the $\{f_k\}$ will either be attained 
for gaussian functions  or by limits obtained using gaussian functions if the value is finite. 
By applying this technique to a characteristic example for the integral given by 
\eqref{eq:BLL} and suggested in part by calculations from \cite{C}, a multilinear 
generalization is obtained for the classic Young's inequality
$$f \rightsquigarrow \varphi * f\ ,\quad 
\| \varphi * f\|_{L^p (\real^n)} \le \|\varphi\|_{L^1(\real^n)} \|f\|_{L^p (\real^n)}$$
where equality for this estimate is never attained for $1<p<\infty$.

\begin{thm}[Multilinear Young's Inequality]
\label{thm:Young-not-attained}
For $f\in L^p (\real^n)$ and $g\in L^q (\real^n)$, $f,g\ge 0$ with $1< p \le m$ and 
$1/p' = (m-1)/2q$
\begin{gather}
\int_{\real^{mn}} \prod f(x_k) \prod_{i<j} g (x_i - x_j) \,dx_1 \cdots  dx_m 
\le C \Big[ \|f\|_{L^p (\real^n)}\Big]^m \Big[ \|g\|_{L^q (\real^n)}\Big]^{m(m-1)/2}
\label{eq:Y-not-attaned}\\
\noalign{\vskip6pt}
C = \sup_{0<\rho } \left[ \frac{\rho^{m/p'}}{(1+m\rho)^{m-1}}\right]^{n/2} 
\left[ p^{m/p} q^{m(m-1)/2q}\right]^{n/2} \nonumber
\end{gather}
If $p=m$, then $C= [m/2^{m-1}]^{n/2}$ and equality is never attained.
\end{thm}

\begin{proof} 
Let $f(x) = e^{-\pi \alpha x^2}$ and $g(x) = e^{-\pi \beta x^2}$, $\alpha,\beta >0$; 
then to calculate the value of the integral on the left-hand side for gaussian functions, 
we need to find the determinant of the matrix of coefficients corresponding to the gaussian 
function
$$\exp \bigg[ -\pi \big( \alpha + (m-1)\beta\big) \sum x_k^2 + 2\pi \beta \sum_{i<j} x_i x_j\bigg]\ .$$
This only requires computing the determinant of the matrix since a rotation applied to the 
matrix would then display eigenvalues on the diagonal corresponding to a symmetric 
matrix and the determinant is invariant under a rotation 
$$\left[ \begin{matrix} \alpha + (m-1)\beta & - \beta & \cdots & -\beta\\
\noalign{\vskip6pt}
-\beta & \alpha + (m-1)\beta&\cdots & -\beta\\
\noalign{\vskip6pt}
\cdot&\cdot&\cdots&\cdot\\
\noalign{\vskip6pt}
-\beta &-\beta&\cdots&\alpha + (m-1)\beta\end{matrix}\right]$$
$m$ variables $x_k \in \real^n$ results in an $m\times m$ matrix. 
To simplify the calculation, consider the matrix
$$\left[\begin{matrix} \delta& -\beta&\cdots & -\beta\\
\noalign{\vskip6pt}
-\beta & \delta &\cdots & -\beta\\
\noalign{\vskip6pt}
\cdot&\cdot&\cdots&\cdot\\
\noalign{\vskip6pt}
-\beta & -\beta & \cdots &\delta \end{matrix}\right]$$
Add all the columns $(k=2,\ldots,m)$ to the first column; 
then subtract the first row from each successive row:
$$\left[ \begin{matrix} 
\delta - (m-1)\beta &-\beta & \cdots & -\beta\\
\noalign{\vskip6pt}
\delta - (m-1)\beta &\delta & \cdots & -\beta\\
\noalign{\vskip6pt}
\cdot&\cdot&\cdots&\cdot\\
\noalign{\vskip6pt}
\delta - (m-1)\beta &-\beta & \cdots & \delta\end{matrix}\right]
\rightsquigarrow 
\left[ \begin{matrix} 
\delta - (m-1)\beta &-\beta & \cdots & -\beta\\
\noalign{\vskip6pt}
0&\delta +\beta&\cdots&0\\
\noalign{\vskip6pt}
\cdot&\cdot&\cdots&\cdot\\
\noalign{\vskip6pt}
0&0&\cdots&\delta +\beta\end{matrix}\right]$$
the determinant equals $[\delta - (m-1)\beta] (\delta +\beta)^{m-1}$ which for the 
variables $(\alpha,\beta)$ becomes $\alpha (\alpha + m\beta)^{m-1}$. 
Then the value of the integral becomes
$$\Big[ \alpha (\alpha +m \beta)^{m-1}\Big]^{-n/2}$$
So the ratio of this integral to the norms on the right-hand side becomes
$$\left[ \frac{\alpha^{m/p} \beta^{m(m-1)/2q}}{\alpha (\alpha +m\beta)^{m-1}}\right]^{n/2}
\mkern-12mu
\left[ p^{m/p} q^{m(m-1)/2q}\right]^{n/2} 
= \left[ \frac{(\beta/\alpha)^{m/p'}}{(1+m\, \beta/\alpha)^{m-1}}\right]^{n/2}
\mkern-12mu
\left[  p^{m/p} q^{m(m-1)/2q}\right]^{n/2} 
\le C$$
This calculation makes explicit the constraint that $1< p\le m$ or equivalently $m/2q \le 1$.
For $m/2q <1$, the optimal constant is attained for $\beta/\alpha = (2q-m)^{-1}$.
\end{proof}
 
This theorem is the formative step in demonstrating the existence of bounds that lead to 
the proof of a multilinear Hardy-Littlewood-Sobolev inequality that is conformally invariant. 
In extending Theorem~\ref{thm:Young-not-attained} 
to allow the input function $g$ to be replaced by the Riesz potential $|x|^{-\lambda}$ with 
$\lambda = n/q$, then $p$ is restricted to lie below the critical index $1<p <m$, as suggested 
by the lack of extremals at that index in the multilinear Young's inequality. 
It is striking that such a natural extension would preserve conformal symmetry. 

\begin{thm}[Multilinear Hardy-Littlewood-Sobolev Inequality]
\label{thm:M HLS inequality}
For $f\in L^p (\real^n)$, $f\ge 0$, with $1<p<m$ and $2n/p' = (m-1)\gamma$, $\gamma=n/q$ 
\begin{gather}
\int_{\real^{mn}} \prod f(x_k) \prod_{i<j} |x_i - x_j|^{-\gamma} dx_1 \cdots dx_m 
\le C\Big[ \| f\|_{L^p (\real^n)}\Big]^m \label{eq:HLS inequal}\\
\noalign{\vskip6pt}
C = \Big[ (4\pi)^{n/2} \Gamma (n/2)\big/ \Gamma (n)\Big]^{m/p'} 
\int_{S^n \times \cdots\times S^n} \prod_{i<j} |\xi_i - \xi_j|^{-\gamma} 
d\xi_1 \cdots d\xi_m\nonumber
\end{gather}
and extremal functions are given by $f(x) = A (1+|x|^2)^{-n/p}$ up to 
conformal automorphism.
Here $d\xi$ is normalized surface measure.
\end{thm}

\begin{proof} 
This result without the sharp constant and in one dimension is given in \cite{C} (see section~2) 
using an interpolation argument. 
Observe that the condition $0<\gamma < 2n/m$ follows from the restriction on the 
index $1 < p< m$.
The intertwined singularity structure can be viewed as an obstruction in determining 
conditions for the integral to be bounded. 
Here rearrangement, dilation invariance,   conformal symmetry and 
axial representation are applied explicitly 
to obtain the inequality, and the background role of the result from 
Theorem~\ref{thm:Young-not-attained} guarantees integrability of the singularity.

The nature of the proof follows the classic steps from the original arguments given 
to obtain sharp constants for the Hardy-Littlewood-Sobolev inequality and the 
Moser-Trudinger inequality (see \cite{Beckner-Annals93}, \cite{Beckner-Essays}, 
\cite{Beckner-JFAA}, and \cite{Lieb83}): 
existence of bounds, realization with inversion symmetry on the multiplicative group 
$\real_+$, existence of extremal functions, determination of optimal constants using 
spherical symmetry, and $SL(2,\real)$ invariance on hyperbolic space. 

For determination of bounds, the proof will be developed in two parts --- first for $n=1$, 
and then $n>1$ where integration over the sphere $S^{n-1}$ comes into play. 
By applying the Brascamp-Lieb-Luttinger rearrangement inequality, the estimate is 
improved when $f$ is a radial decreasing function. 
Then set $h(y) = |x|^{n/p} f(x)$ where $y = |x|$; notice that $h$ is bounded so that 
$h\in L^p (\real_+) \cap L^\infty (\real_+)$. 
The resulting inequality becomes 
\begin{gather} 
\int_{(\real_+)^m \times (S^{n-1})^m} 
\prod h (y_k) \prod_{i<j} \left[ \frac{y_i}{y_j} + \frac{y_j}{y_i} - 2\xi_i \cdot \xi_j\right]^{-\gamma/2}
\frac{dy_1}{y_1} \cdots \frac{dy_m}{y_m} \ d\xi_1 \cdots d\xi_m \nonumber\\
\noalign{\vskip6pt}
\le C\Big[ \|h\|_{L^p (\real_+)}\Big]^m 
\label{eq:radial decreasing}
\end{gather}
Because the ``potentials'' are now symmetric decreasing away from the origin $\{y=1\}$, 
the Brascamp-Lieb-Luttinger inequality can be applied to improve the inequality 
so that $h(1/y) = h(y)$ and $h(y)$ is monotone decreasing for $y>1$. 
This step then implies that inequality \eqref{eq:HLS inequal} is improved if 
(1)~$f$ is radial decreasing, 
(2)~$|x|^{n/p} f(x)$ is decreasing for $|x|>1$, 
and (3)~$f(|x|^{-1}) = |x|^{2n/p} f(|x|)$. 
These conditions are precisely what is meant  by saying that $f$ possesses an ``inversion 
symmetry.'' 
Set $y= e^t$; then \eqref{eq:radial decreasing} becomes
\begin{gather}
2^{-mn/2p'} \int_{\real^m \times (S^{n-1})^m} 
\prod h(t_k) \prod_{i<j} \Big[ \cosh (t_i - t_j) - \xi_i \cdot \xi_j\Big]^{-\gamma/2} 
dt_1 \cdots dt_m\ d\xi_1 \cdots d\xi_m \nonumber\\
\le C\Big[ \|h\|_{L^p (\real)}\Big]^m   \label{eq:radial decrease2} 
\end{gather}
Here the objective first is to show existence of a bound and second to show 
existence of extremals.
This theme --- to show first bounds and then second existence of limits --- 
in many essential ways characterizes the objectives in Stein's classic text on 
{\em Singular integrals} (\cite{Stein70}). 
Observe that the function $h$ on $\real$ inherits the bounds of $h$ on $\real_+$
(a simple change of variables) and these bounds for $h\in L^p (\real) \cap L^\infty (\real)$ 
are controlled by the value $\|f\|_{L^p (\real^n)}$. 
For the one-dimensional setting where $S^{n-1} = \{\pm 1\}$, 
the simplest estimate is that 
$$\Big[ \cosh (t_i - t_j) - \xi_i \xi_j\Big]^{-\gamma/2} 
\le \Big[ \cosh (t_i - t_j) - 1\Big]^{-1/(p'(m-1))}$$
so the inequality requires an estimate for 
\begin{equation}\label{eq:Stein} 
2^{m(1- 1/2p')} \int_{\real^m} \prod h (t_k) \prod_{i<j} 
\Big[ \cosh (t_i - t_j)-1\Big]^{-\gamma/2}  dt_1 \cdots dt_m 
\end{equation}
Since $1< p< m$, choose $p < p_* <m$; since $h\in L^p (\real)\cap L^\infty (\real)$, 
$h\in L^{p_*} (\real)$. 
Define $q_*$ by $1/p'_* = (m-1)/2q_*$; since $p'_* < p'$, then $q_* < q$. 
Now claim that $g(t) = (\cosh t-1)^{-\gamma/2} \in L^{q_*} (\real)$ since 
$\gamma q_* = p'_* /p' <1$ so that by the Multilinear Young's Inequality 
(Theorem~\ref{thm:Young-not-attained}), the integral given by \eqref{eq:Stein} is 
bounded in terms of $[\|f\|_{L^p (\real)}]^m$. 

To address the question of boundedness in higher dimensions for 
$$\Lambda (h) = \int_{\real^m \times (S^{n-1})^m} 
\prod h(t_k) \prod_{i<j} \Big[ \cosh (t_i - t_j) - \xi_i \cdot \xi_j\Big]^{-\gamma/2} 
dt_1 \cdots dt_m\ d\xi_1 \cdots d\xi_m$$
observe that 
$$\Big[ \cosh (t_i - t_j) - \xi_i \cdot \xi_j\Big]^{-n/p'(m-1)} 
\le A \Big[ \cos (t_i - t_j) -1\Big]^{-1/p'(m-1)} 
\Big[ 1- \xi_i \cdot \xi_j\Big]^{-(n-1)/p'(m-1)}$$
Then the integration splits and the ``$t$-integration'' is handled as in the one-dimensional 
case above. 
The integral over the $(n-1)$-dimensional sphere is bounded by again applying 
Theorem~\ref{thm:Young-not-attained}. 
\end{proof}

\begin{Lem} 
For $1< p < m$ and $2N/p' = (m-1)\gamma$, $\gamma = N/q$, $N\ge 1$ the following 
integral is bounded 
\begin{equation}\label{Lem:bounded} 
\int_{S^N\times\cdots\times S^N} \prod_{i<j} |\xi_i - \xi_j|^{-\gamma} 
d\xi_1 \cdots d\xi_m \le C
\end{equation}
\end{Lem}

\begin{proof} 
Apply Theorem~\ref{thm:Young-not-attained} in the case where $f$ is constant with 
bounded support for the index $p_*$, $p< p_* <m$. 
Then $|\xi-\eta|^{-\gamma}$ as a function of $\xi$ is in $L^{q^*} (S^N)$ where
$2/p'_* = (m-1)/q_*$ and $\gamma q^* < N$.
\end{proof}

Now having determined that the integral $\Lambda (h)$ is bounded for $\|h\|_{L^p(\real^*)}=1$, 
it follows directly that an extremal function exists for this functional where the ``best constant''
$C$ is attained (see the nature of the argument given on page~40 in \cite{Beckner-Essays}).
Choose a sequence $\{ h_k\}$ with $\|h_k\|=1$ so that $\Lambda (h_k) \rightsquigarrow C$; 
the $h_k$ will be symmetric decreasing and uniformly bounded from the earlier application 
of the Brascamp-Lieb-Luttinger rearrangement inequality. 
Since the functions $\{h_k\}$ are decreasing, the Helly selection principle can be applied 
to choose a subsequence that converges almost everywhere to a function $h\in L^p(\real)$ 
with $\|h\|_p \le 1$ and $\Lambda (h) \le C$.
Since $h_k$ has a uniform functional bound 
$$h_k(t) \le \Lambda \Big(1+ |t|^{n/p}\Big)^{-1} = u(t)$$
and the argument used above to show that the integral given by \eqref{eq:Stein} is bounded 
in terms of $L^p$ norms now shows that $\Lambda (u)$ is bounded and hence the 
integrand for the integrals $\Lambda (h_k)$ has a uniform $L^1 (\real^m \times (S^{n-1}))$ 
majorant so that the dominated convergence theorem implies that 
$$\Lambda (h_k) \longrightarrow \Lambda (h)\ ,\qquad k\to \infty$$
and hence $\Lambda (h) =C$ and $h$ is an extremal function which is symmetric decreasing. 
Tracing back through successive changes of variables, this exhibits a radial decreasing 
extremal function $f$ for inequality \eqref{eq:HLS inequal} which possesses the 
inversion symmetry 
\begin{equation}\label{eq:tracing back} 
f\big(|x|^{-1}\big) = |x|^{2n/p} f(|x|)\ .
\end{equation}

The next objective is to characterize the extremal functions, and show that up to conformal 
automorphism an extremal function for inequality \eqref{eq:HLS inequal} is given by 
conformal factors 
$$f(x) = \Lambda \big(1+ |x|^2\big)^{-n/p}\ .$$
Conformal invariance allows the inequality to be transferred to the sphere $S^n$. 
The equivalence between $\real^n$ and $S^n -\{\text{pole}\}$ is given by:
\begin{gather*}
\xi = \left( \frac{1-|x|^2}{1+|x|^2} , \frac{2x}{1+|x|^2} \right)\ ,\qquad 
d\xi = \pi^{-n/2} \left[ \frac{\Gamma (n)}{\Gamma (n/2)}\right] \big(1+|x|^2\big)^{-n}\, dx\\
\noalign{\vskip6pt}
|x-y| = \frac12 |\xi-\eta| \Big[ \big(1+ |x|^2\big) \big(1+ |y|^2\big)\Big]^{1/2}
\end{gather*}
where $d\xi$ corresponds to normalized surface measure on the sphere. 
By setting $F(\xi) = (1+ |x|^2)^{n/p} f(x)$, then an inequality on $S^n$ equivalent 
to \eqref{eq:HLS inequal} becomes: 
\begin{gather} 
\int_{S^n\times S^n} \prod F(\xi_k) \prod_{i<j} |\xi_i - \xi_j|^{-\gamma}\, d\xi_1\cdots d\xi_m
\le D\Big[ \|F\|_{L^p (S^n)}\Big]^m \label{eq:HLS inequal2}\\
\noalign{\vskip6pt}
D = C\Big[ (4\pi)^{-n/2} \Gamma(n) \,\big\slash\, \Gamma (n/2)\Big]^{m/p'}\nonumber
\end{gather}
So if extremal functions are given by $(1+|x|^2)^{-n/p}$, then the constant $C$ for 
Theorem~\ref{thm:M HLS inequality} is given by 
$$C = \Big[ (4\pi^{n/2} \Gamma (n/2) \,\big\slash\, \Gamma (n)\Big]^{m/p'} 
\int_{S^n\times \cdots \times S^n} 
\prod_{i<j} |\xi_i - \xi_j|^{-\gamma} d\xi_1 \cdots d\xi_m\ .$$

To demonstrate that extremal functions for $n\ge 2$ are given by $A (1+|x|^2)^{-n/p}$ up 
to conformal automorphism, an equivalent realization for the multilinear 
Hardy-Littlewood-Sobolev inequality on hyperbolic space is used. 
It's very natural to recognize the conformal equivalence for the classical  geometries 
--- the plane, the sphere and the two-sheeted hyperboloid, the latter corresponding to 
the homogeneous space for the group $SO(n,1)$. 
But the equivalence for the standard Liouville-Beltrami upper half-space model 
for hyperbolic space represents a more subtle dilation invariance which utilizes axial symmetry.
More striking is that the Hardy-Littlewood-Sobolev inequality on $\real^n$ for $n\ge 2$ 
(Theorem~\ref{thm:M HLS inequality}) 
has an equivalent formulation on the hyperbolic space $\HH^\ell$ for $2\le \ell\le n$ 
which depends on the dimension of the submanifold used for the axial symmetry. 

The Liouville-Beltrami upper half-space model includes the Lie group structure of 
affine mappings and incorporates ``axial symmetry'' from the Euclidean picture. 
Here 
$$
w = (x,y) \in \real_+^\ell \simeq \real^{\ell-1} \times \real^+ \simeq SO(\ell,1) \,\big\slash\, 
SO(\ell) \simeq \HH^\ell$$
with the Poincar\'e distance
$$d(w,w') = \frac{|w-w'|}{2\sqrt{yy'}}$$
and left-invariant Haar measure $d\nu = y^{-\ell} \,dy\,dx$. 
The Riemannian metric on $\HH^\ell$ is defined by
$$ds^2 = y^{-2} (dx^2 + dy^2)$$ 
and the invariant gradient is given by $D = y\nabla$. 
The group structure of hyperbolic space corresponds to a non-unimodular Lie group 
that is an extension of the affine ``$ax+b$ group.'' 
Hyperbolic space $\HH^\ell$ is identified with the subgroup of $SL (\ell,\real)$ given by 
all matrices of the form 
$$\root \ell \of y \left(\begin{matrix} \text{\bf 1}& x/y\\ 
\noalign{\vskip6pt}
0&1/y\end{matrix}\right)$$
where $x\in \real^{\ell-1}$ is represented by a column vector and $y>0$. 
Such matrices can act via fractional linear transformation on $\real_+^\ell \simeq \HH^\ell$
$$w = x+iy\xi \in \real_+^\ell \longrightarrow \frac{Aw+B}{Cw+D}$$
for a matrix $\Big( {A\atop C}\ {B\atop D}\Big)$ 
where $A = (\ell-1)\times (\ell-1)$ matrix, 
$B = (\ell-1) \times 1$ matrix, 
$C = 1\times (\ell-1)$ matrix, 
$D = |x|$ matrix and fixed non-zero $\xi \in \real^{\ell-1}$. 
The group action here corresponds to the multiplication rule 
$$(x,y) (u,v) = (x+yu, yv)$$
for $x,u \in \real^{\ell-1}$ and $y,v >0$ so this is again a non-unimodular group with 
the modular function $\Delta (x,y) = y^{-(\ell-1)}$. 
The group identity is $\widehat 0  = (0,1)$.

Following the argument given in \cite{Beckner-JFAA}, consider non-negative 
radial functions $f$ in equation~\eqref{eq:HLS inequal} and set 
$y= |x'|$ where 
$x\in \real^n = (u,x') \rightsquigarrow
w = (u,y) \in \real^{\ell-1} \times \real_+ \simeq \HH^\ell$. 
Let $F(u,y) = y^{n/p} f(u,x')$; then 
$$\|f\|_{L^p (\real^n)} 
= \left[ \frac{2\pi^{n-\ell+1}}{\Gamma [(n-\ell+1)/2]}\right]^{1/p} \|F\|_{L^p (\HH^\ell)}$$
and equation~\eqref{eq:HLS inequal} 
results in the family of estimates for $2\le \ell \le n$;
\begin{gather}
\int_{(\HH^\ell)^m} \prod F(w_k) \int_{(S^{n-\ell})^m} 
\prod_{i<j} \left[ d^2 (w_i,w_j) + (1-\xi_i \cdot \xi_j)/2\right]^{-\gamma/2} 
d\xi_1 \cdots d\xi_m\ d\nu_1 \cdots d\nu_m\nonumber\\
\le C\left[ \frac{\pi^{n-\ell+1}}{2^{n-1} \Gamma [(n-\ell+1)/2]}\right]^{-m/p'} 
\Big[ \|F\|_{L^p (\HH^\ell)} \Big]^m
\label{eq:family}
\end{gather}
where $C$ is the optimal constant for equation~\eqref{eq:HLS inequal}. 
The multiplicity of ``fractional integral'' inequalities with varied parameters corresponds 
to earlier results obtained for embedding estimates and Sobolev inequalities on 
hyperbolic space (see \cite{Beckner-JFAA}, \cite{Beckner-PAMS01}). 

The Brascamp-Lieb-Luttinger rearrangement inequality and the nature of the 
multiplication rule for group action on $\HH^\ell$ result in a Riesz-Sobolev rearrangement 
inequality on hyperbolic space. 
This general result is presented here in the reduced form that is needed for 
determining the form of the extremal functions for Theorem~\ref{thm:M HLS inequality}.

\begin{Lem}[Riesz-Sobolev Inequality on Hyperbolic Space] 
Let $\{ F_k\}$ be a sequence of non-negative functions on $\HH^n$ and $\{g_{ij}\}$ 
a sequence of non-negative decreasing functions on $\real_+$. 
Then 
\begin{gather} 
\int_{(\HH^n)^m} \prod F_k (w_k) \prod_{i<j} g_{ij} \Big[ d(w_i,w_j)\Big] \,dw_1 \cdots dw_m
\hskip1truein
\nonumber\\
\le \int_{(\HH^n)^m} \prod F_k^* (w-k) \prod_{i<j} g_{ij} \Big[ d(w_i,w_j)\Big] \, dw_1\cdots dw_m
\label{eq:lem-RS}
\end{gather}
where $F^*$ corresponds to the equimeasurable radial  decreasing rearrangement on 
$\HH^n$ of $F$ with respect to the distance function $d(w,\widehat 0\,)$ where for 
$w = (u,y)$
$$1+ 2d^2 (w,\widehat 0\,) = \frac{y^2 +u^2 +1}{2y}\ .$$
\end{Lem}

\begin{proof} 
Apply the Brascamp-Lieb-Luttinger rearrangement inequality to the integrations over
the Euclidean variables $\{u_k\}$. 
Use the Lusternik technique to take a sequence of successive rearrangements followed 
by a change of variables so that the sequence of rearranged functions converges 
to the resulting equimeasurable radial decreasing rearrangements of the functions $\{F_k\}$.
\end{proof} 

Apply this Lemma to the left side of equation~\eqref{eq:family} and follow the arguemnt 
back to equation~\eqref{eq:HLS inequal}.
Then there must exist an extremal $f_{\#}$ of the form 
$$f_{\#} (u,y) = y^{-n/p} F^* \left( \frac{y^2 +u^2 +1}{2y}\right)\ ,\qquad y = |x'|$$ 
up to conformal action, the only possible form for $f_{\#}$ is to be radial.
Hence 
$$f_{\#} (x) = f_{\#} (u,y) = y^{-n/p} F^* \left( \frac{y^2 +u^2 +1}{2y}\right)$$
and 
$$f_{\#} (x) = f_{\#} (u,y) = A  ( 1+u^2 + y^2  )^{-n/p} 
= A (1+|x|^2)^{-n/p}\ .$$
This completes proof of Theorem~\ref{thm:M HLS inequality}
when the dimension $n$ is at least two. 
The argument is important because it illustrates how hyperbolic symmetry is embedded 
in the conformal structure of the Riesz functional and the Hardy-Littlewood-Sobolev 
inequality corresponding to Theorem~\ref{thm:M HLS inequality} and 
equation~\eqref{eq:HLS inequal}.
The competing radial and cylindrical symmetry force the extremal to be of the form 
$f(x) = A(1+|x|^2)^{-n/p}$ up to conformal automorphism.
For the one-dimensional case where ``axial symmetry'' is not available, the observation 
by Carlen and Loss that conformal transformations can be chosen that break geodesic 
symmetry so that a sequence of successive rearrangements followed by conformal 
automorphisms will generate an extremal function in the limit (see \cite{CL} 
and page~42 in \cite{Beckner-Essays}) 
can be used to fully complete the proof of Theorem~\ref{thm:M HLS inequality}. 
While simple and elegant, such an argument does not bring out the geometric 
symmetry that is encompassed by the conformal structure. 
Additionally, the arguments here can be used for the case of multiple functions 
$\{f_k\}$, $k=1,\ldots,m$ in Theorem~\ref{thm:M HLS inequality} 
rather than a single function $f$.\qed

\begin{Rem} 
It's perhaps useful in terms of possible application to quantum dynamical processes 
to consider Coulomb potentials and many body interactions. 
For low dimensions, interesting cases are: $n=3$, $\gamma=1$, $m=4$ and 
$p=2$; $n=4$, $\gamma=2$, $m=3$ and $p=2$.
\end{Rem}

The framework for these arguments can be extended to include multilinear Stein-Weiss
fractional integrals though without the capability to compute sharp constants and extremal 
functions.
The argument including existence for extremals follows from adapting the mechanism of 
the proof used in Theorem~\ref{thm:M HLS inequality}. 
A variant of this result in one dimension is given in \cite{C}. 

\begin{thm}[Multilinear Stein-Weiss Inequality]
\label{thm:Stein-Weiss} 
For $f\in L^p (\real^n)$, $f\ge 0$, $1<p \le m$, $\beta >0$ and $2n/p' = 2\beta + (m-1)\gamma$, 
$\gamma = n/q >0$
\begin{equation}\label{eq:Stein-Weiss} 
\int_{\real^{mn}} \prod \left[ f(x_k) |x_k|^{-\beta}\right] 
\prod_{i<j} |x_i - x_j|^{-\gamma} dx_1 \cdots dx_m 
\le \Big[ \|f\|_{L^p (\real^n)}\Big]^m
\end{equation}
In addition, extremal functions exist for $1<p<m$.
\end{thm}

\begin{Cor}
For $2\le p\le m$, $1/p + 1/p' = 1$
\begin{equation}\label{eq:cor-Stein Weiss}
\Big| \int_{\real^{mn}} \prod \widehat f (x_k) \prod_{i<j} |x_i - x_j|^{-\gamma} 
dx_1 \cdots dx_m\Big| 
\le C\Big[ \| (-\Delta/4\pi^2)^{\beta/2} f \|_{L^{p'} (\real^n)}\Big]^m
\end{equation}
where $[(-\Delta/4\pi^2)^{\beta/2} f]^\wedge = |x|^\beta \widehat f$.
\end{Cor}

\begin{proof} 
Apply the Brascamp-Lieb-Luttinger rearrangement inequality for the integral that 
appears in equation~\eqref{eq:Stein-Weiss} so that the function $f$ may be taken to 
be radial decreasing. 
Then follow the development of the proof as for Theorem~\ref{thm:M HLS inequality} 
and the inequality \eqref{eq:radial decrease2}. 
Here one obtains the similar functional integral
$$\int_{\real^m \times (S^{n-1})^m} \prod h (t_k) \prod_{i<j} 
\Big[ \cosh (t_i - t_j) - \xi_i \cdot \xi_j\Big]^{-\gamma/2} dt_1 \cdots dt_m \ 
d\xi_1 \cdots d\xi_m$$
and since $(m-1)\gamma < 2n/p'$ due to $\beta >0$, application of the multilinear 
Young's inequality (Theorem~\ref{thm:Young-not-attained}) will show that the functional 
is bounded for $1<p\le m$ and allow the Helly selection principle to show the existence 
of an extremal function where the optimal constant is attained for the index range $1<p<m$.
The corollary is a direct restatement of the theorem to demonstrate the connection with 
embedding.
\end{proof}


\section*{Acknowledgements}

I would like to thank Guozhen Lu for his warm encouragement.


\end{document}